\documentclass[12pt]{amsart}
\usepackage{amsmath,amssymb,amsbsy,amsfonts,amsthm,latexsym,mathabx,
            amsopn,amstext,amsxtra,euscript,amscd,stmaryrd,mathrsfs,
            cite,array,mathtools,enumerate}
\usepackage{url}
\usepackage[colorlinks,linkcolor=blue,anchorcolor=blue,citecolor=blue]{hyperref}

\usepackage{color}
\usepackage{float}

\hypersetup{breaklinks=true}

\begin{document}

\newtheorem{theorem}{Theorem}
\newtheorem{lemma}[theorem]{Lemma}
\newtheorem{claim}[theorem]{Claim}
\newtheorem{cor}[theorem]{Corollary}
\newtheorem{proposition}[theorem]{Proposition}
\newtheorem{definition}{Definition}
\newtheorem{question}[theorem]{Question}
\newtheorem{remark}[theorem]{Remark}
\newcommand{\hh}{{{\mathrm h}}}

\numberwithin{equation}{section}
\numberwithin{theorem}{section}
\numberwithin{table}{section}

\def\sssum{\mathop{\sum\!\sum\!\sum}}
\def\ssum{\mathop{\sum\ldots \sum}}
\def\dsum{\mathop{\sum \sum}}
\def\iint{\mathop{\int\ldots \int}}

\def\squareforqed{\hbox{\rlap{$\sqcap$}$\sqcup$}}
\def\qed{\ifmmode\squareforqed\else{\unskip\nobreak\hfil
\penalty50\hskip1em\null\nobreak\hfil\squareforqed
\parfillskip=0pt\finalhyphendemerits=0\endgraf}\fi}

\newfont{\teneufm}{eufm10}
\newfont{\seveneufm}{eufm7}
\newfont{\fiveeufm}{eufm5}
%
%
\newfam\eufmfam
     \textfont\eufmfam=\teneufm
\scriptfont\eufmfam=\seveneufm
     \scriptscriptfont\eufmfam=\fiveeufm
%
%
\def\frak#1{{\fam\eufmfam\relax#1}}

\newcommand{\bflambda}{{\boldsymbol{\lambda}}}
\newcommand{\bfmu}{{\boldsymbol{\mu}}}
\newcommand{\bfxi}{{\boldsymbol{\xi}}}
\newcommand{\bfrho}{{\boldsymbol{\rho}}}

\def\fK{\mathfrak K}
\def\fT{\mathfrak{T}}

\def\fA{{\mathfrak A}}
\def\fB{{\mathfrak B}}
\def\fC{{\mathfrak C}}

\def\E{\mathsf {E}}

\def \balpha{\bm{\alpha}}
\def \bbeta{\bm{\beta}}
\def \bgamma{\bm{\gamma}}
\def \blambda{\bm{\lambda}}
\def \bchi{\bm{\chi}}
\def \bphi{\bm{\varphi}}
\def \bpsi{\bm{\psi}}

\def\eqref#1{(\ref{#1})}

\def\vec#1{\mathbf{#1}}


\def\cA{{\mathcal A}}
\def\cB{{\mathcal B}}
\def\cC{{\mathcal C}}
\def\cD{{\mathcal D}}
\def\cE{{\mathcal E}}
\def\cF{{\mathcal F}}
\def\cG{{\mathcal G}}
\def\cH{{\mathcal H}}
\def\cI{{\mathcal I}}
\def\cJ{{\mathcal J}}
\def\cK{{\mathcal K}}
\def\cL{{\mathcal L}}
\def\cM{{\mathcal M}}
\def\cN{{\mathcal N}}
\def\cO{{\mathcal O}}
\def\cP{{\mathcal P}}
\def\cQ{{\mathcal Q}}
\def\cR{{\mathcal R}}
\def\cS{{\mathcal S}}
\def\cT{{\mathcal T}}
\def\cU{{\mathcal U}}
\def\cV{{\mathcal V}}
\def\cW{{\mathcal W}}
\def\cX{{\mathcal X}}
\def\cY{{\mathcal Y}}
\def\cZ{{\mathcal Z}}
\newcommand{\rmod}[1]{\: \mbox{mod} \: #1}

\def\cg{{\mathcal g}}

\def\e{{\mathbf{\,e}}}
\def\ep{{\mathbf{\,e}}_p}
\def\eq{{\mathbf{\,e}}_q}

\def\em{{\mathbf{\,e}}_m}

\def\Tr{{\mathrm{Tr}}}
\def\Nm{{\mathrm{Nm}}}

\def\rE{{\mathrm{E}}}
\def\rT{{\mathrm{T}}}

 \def\SS{{\mathbf{S}}}

\def\lcm{{\mathrm{lcm}}}

\def\t{\tilde}
\def\ov{\overline}

\def\({\left(}
\def\){\right)}
\def\l|{\left|}
\def\r|{\right|}
\def\fl#1{\left\lfloor#1\right\rfloor}
\def\rf#1{\left\lceil#1\right\rceil}
\def\flq#1{\langle #1 \rangle_q}

\def\mand{\qquad \mbox{and} \qquad}

\newcommand{\commIg}[1]{\marginpar{%
\begin{color}{magenta}
\vskip-\baselineskip 
\raggedright\footnotesize
\itshape\hrule \smallskip Ig: #1\par\smallskip\hrule\end{color}}}

\newcommand{\commSi}[1]{\marginpar{%
\begin{color}{blue}
\vskip-\baselineskip 
\raggedright\footnotesize
\itshape\hrule \smallskip Si: #1\par\smallskip\hrule\end{color}}}




\hyphenation{re-pub-lished}

\mathsurround=1pt

\def\bfdefault{b}
\overfullrule=5pt

\def \F{{\mathbb F}}
\def \K{{\mathbb K}}
\def \Z{{\mathbb Z}}
\def \Q{{\mathbb Q}}
\def \R{{\mathbb R}}
\def \C{{\\mathbb C}}
\def\Fp{\F_p}
\def \fp{\Fp^*}

\def\Smn{S_{k,\ell,q}(m,n)}

\def\Kmn{\cK_p(m,n)}
\def\psmn{\psi_p(m,n)}

\def\SM{\cS_{k,\ell,q}(\cM)}
\def\SMN{\cS_{k,\ell,q}(\cM,\cN)}
\def\SAMN{\cS_{k,\ell,q}(\cA;\cM,\cN)}
\def\SABMN{\cS_{k,\ell,q}(\cA,\cB;\cM,\cN)}

\def\SIJq{\cS_{k,\ell,q}(\cI,\cJ)}
\def\SAJq{\cS_{k,\ell,q}(\cA;\cJ)}
\def\SABJq{\cS_{k,\ell,q}(\cA, \cB;\cJ)}

\def\sM{\cS_{k,q}^*(\cM)}
\def\sMN{\cS_{k,q}^*(\cM,\cN)}
\def\sAMN{\cS_{k,q}^*(\cA;\cM,\cN)}
\def\sABMN{\cS_{k,q}^*(\cA,\cB;\cM,\cN)}

\def\sIJq{\cS_{k,q}^*(\cI,\cJ)}
\def\sAJq{\cS_{k,q}^*(\cA;\cJ)}
\def\sABJq{\cS_{k,q}^*(\cA, \cB;\cJ)}
\def\sABJp{\cS_{k,p}^*(\cA, \cB;\cJ)}

 \def \xbar{\overline x}

\author[S.  Macourt] {Simon Macourt}
\address{Department of Pure Mathematics, University of New South Wales,
Sydney, NSW 2052, Australia}
\email{s.macourt@student.unsw.edu.au}

\title[Incidence Results and Bounds of Exponential Sums]{Incidence Results and Bounds Of Trilinear and Quadrilinear Exponential Sums}

\begin{abstract} We give a new bound on the number of collinear triples for two arbitrary subsets of a finite field. This improves on existing results which rely on the Cauchy inequality. We then use this to provide a new bound on trilinear and quadrilinear exponential sums. 
 \end{abstract}

\maketitle

\section{Introduction}
\subsection{Set Up}
For a prime $p$, we define $\F_p$ to be the finite field of $p$ elements. We also let $\F^*_p=\F_p/ \{0\}$. 
We define the line 
\begin{align*}
\ell_{a,b}=\{(x,y)\in \F^2_p : y=ax+b\}
\end{align*}
for some $(a,b)\in \F^2_p$. We let $\cA,\cB \subseteq \F_p$, with $|\cA|=A$, $|\cB|=B$ and $A\le B$. We also define the number of incidences of any line with $\cA \times \cB$ to be
\begin{align*}
\iota_{\cA\times \cB}(\ell_{a,b}) = |\{(\cA\times \cB) \cap \ell_{a,b}\} |.
\end{align*}
Furthermore for $\lambda, \mu \in \F^*_p$, we define the number of collinear triples  $T_{\lambda, \mu}(\cA,\cB)$ to be the number of solutions to 
\begin{align*}
(a_1-\lambda a_2)(b_1-\mu b_2)=(a_1-\lambda a_3)(b_1-\mu b_3), \ \  a_i \in \cA, b_i \in \cB, i=1,2,3.
\end{align*}
We define $T_{1,1}(\cA,\cB)=T(\cA,\cB)$ and for $\cA=\cB$ we define $T(\cA,\cA)=T(\cA)$. 

We also define the weighted trilinear exponential sums over sets \linebreak[4] $\cX, \cY, \cZ \subset \F_p$
\begin{align*}
T(\cX,\cY,\cZ;\rho,\sigma,\tau)=\sum_{x \in\cX} \sum_{y \in \cY} \sum_{z \in \cZ} \rho_{x,y} \sigma_{x,z} \tau_{y,z} \ep(a xyz),
\end{align*}
where $a\in \F^*_p$ and $\rho_{x,y}, \sigma_{x,z}, \tau_{y,z}$ are 2-dimensional weights that are bounded by 1.

Similarly, we define the weighted quadrilinear exponential sums over sets  $\cW, \cX, \cY, \cZ \subset \F_p$
\begin{align*}
&T(\cW,\cX,\cY,\cZ;\vartheta,\rho,\sigma,\tau)\\
&\qquad \qquad=\sum_{w \in\cW}\sum_{x \in\cX} \sum_{y \in \cY} \sum_{z \in \cZ} \vartheta_{w,x,y}\rho_{w,x,z} \sigma_{w,y,z} \tau_{x,y,z} \ep(a wxyz),
\end{align*}
where $a\in \F^*_p$ and $\vartheta_{w,x,y}, \rho_{w,x,z}, \sigma_{w,y,z}, \tau_{x,y,z}$ are 3-dimensional weights that are bounded by 1.

Throughout the paper we use the notation $A \ll B$ to indicate \linebreak[4] $|A| \le c|B|$ for some absolute constant $c$.
\subsection{New Results}
Our main result is the following theorem on the number of collinear triples.
\begin{theorem} \label{thm:TAB}
Let $\cA, \cB \subset \F_p$ with $|\cA|=A\le |\cB|=B$ and $\lambda, \mu \in \F^*_p$. Then
\begin{align*}
T_{\lambda, \mu}(\cA,\cB) - \frac{A^3B^3}{p} \ll p^{1/2}A^2B^{3/2}+AB^3.
\end{align*}
\end{theorem}
Using Lemma \ref{lem:D*}, which comes as a consequence of Theorem \ref{thm:TAB}, we provide the following new bounds on trilinear and quadrilinear exponential sums.
\begin{theorem}\label{thm:tri}
Let $\cX, \cY, \cZ \subset \F_p$ with $|\cX|=X, |\cY|=Y, |\cZ|=Z$, and $X\ge Y\ge Z$. Then,
\begin{align*}
T(\cX,\cY,\cZ;\rho,\sigma,\tau) \ll p^{3/16}X^{13/16}Y^{7/8}Z^{7/8}.
\end{align*}
\end{theorem}
We compare the above the result with previous bounds in the following section. As an example, in the special case where $X=Y=Z$ the bound from Theorem \ref{thm:tri} is stronger than previous results for $p^{1/2}<X<p^{5/9}$.
\begin{theorem} \label{thm:quad}
Let $\cW, \cX, \cY, \cZ \subset \F^*_p$ with $|\cW|=W, |\cX|=X, |\cY|=Y$, $|\cZ|=Z$ and $W\ge X\ge Y\ge Z$. Then,
\begin{align*}
&T(\cW,\cX,\cY,\cZ;\vartheta,\rho,\sigma,\tau) \ll  p^{3/32}W^{29/32}X^{15/16}Y^{15/16}Z^{31/32}.
\end{align*}
\end{theorem}
Again, we give an example of when our bound is non-trivial by considering the special case  $W=X=Y=Z$ and note that the bound from Theorem \ref{thm:quad} is stronger than all existing bounds for $p^{1/2}<W<p^{13/24}$.

\subsection{Previous Results}
Recent results on $T(\cA,\cB)$ have been given by using the Cauchy inequality on bounds for $T(\cA)$. For this reason previous bounds for $T(\cA,\cB)$ are symmetric. We compare our result with that of Aksoy Yazici, Murphy, Rudnev and Shkredov \cite[Proposition 5]{AMRS} 
\begin{align*}
T(\cA)\ll \frac{A^6}p +A^{9/2}
\end{align*}
hence, by the Cauchy inequality,
\begin{align*}
T(\cA,\cB)\ll \left(\frac{A^3}{p^{1/2}} +A^{9/4}\right)\left(\frac{B^3}{p^{1/2}} +B^{9/4}\right)+AB^3.
\end{align*}
We see that for $A=B$ our new bound is stronger for $A>p^{1/2}$. More generally, our new bound is stronger when $AB^3>p^2$. We also compare our result to that of Murphy, Petridis, Roche-Newton, Rudnev and Shkredov \cite[Theorem 10]{MPR-NRS}
\begin{align*}
T(\cA)\ll \frac{A^6}p +A^{7/2}p^{1/2}
\end{align*}
hence, by the Cauchy inequality,
\begin{align*}
T(\cA,\cB)\ll \left(\frac{A^3}{p^{1/2}} +A^{7/4}p^{1/4}\right)\left(\frac{B^3}{p^{1/2}} +B^{7/4}p^{1/4}\right)+AB^3.
\end{align*}
We see that our bound is equal to the above result for $A=B$, and stronger otherwise. We also mention that \cite{MaShkShp} gives a bound on collinear triples over subgroups.

Trilinear sums have been estimated by Bourgain and Garaev \cite{BouGar}. Variations and improvements have been made since, see \cite{Bourg1, Bourg3, BouGlib, Gar, Ost}. More recently Petridis and Shparlinski \cite{PetShp} have given new bounds on weighted trilinear and quadrilinear exponential sums. We compare our bound on trilinear sums to \cite[Theorem 1.3]{PetShp}
\begin{align*}
T(\cX,\cY,\cZ;\rho,\sigma,\tau) \ll p^{1/8}X^{7/8}Y^{29/32}Z^{29/32}.
\end{align*}
We see that our new bound, Theorem \ref{thm:tri}, improves that of Petridis and Shparlinski \cite{PetShp} for $XY^{1/2}Z^{1/2}\ge p$.
Our bound  from Theorem \ref{thm:tri} is stronger than that of the triangle inequality
\begin{align*}
T(\cX,\cY,\cZ;\rho,\sigma,\tau) \ll XYZ
\end{align*}
for $XY^{2/3}Z^{2/3} > p$. Similarly, it is stronger than the classical bound on bilinear sums, from Lemma \ref{lem:bilin},
\begin{align*}
T(\cX,\cY,\cZ;\rho,\sigma,\tau) \ll p^{1/2}X^{1/2}Y^{1/2}Z
\end{align*}
for $XY^{6/5}Z^{-2/5} \le p$. Letting $X=Y=Z$ we see that under these conditions Theorem \ref{thm:tri} is stronger than previous bounds for $p^{1/2}<X<p^{5/9}$. We give another example for when our bound is non-trivial. Setting $X=p^{2/3}, Y=Z=p^{2/5}$ we obtain from Theorem \ref{thm:tri}
$$
T(\cX,\cY,\cZ;\rho,\sigma,\tau) \ll p^{343/240}= XYZp^{-3/80}.
$$
One can easily compare this with results from previous bounds and see our new bound is stronger. We also mention that our bound is strongest for $X$ much larger than $Y$. We finally mention the bound on unweighted trilinear sums due to Garaev \cite{Gar}. We note that when our bound is stronger than that of Shparlinski and Petridis \cite{PetShp}, it also outperforms that of Garaev \cite{Gar}.

Similarly, we compare our results on quadrilinear exponential sums to \cite[Theorem 1.4]{PetShp} 
\begin{align}\label{eq:quadPetShp}
T(\cW, \cX,\cY,\cZ;\vartheta, \rho,\sigma,\tau) \ll p^{1/16}W^{15/16}(XY)^{61/64}Z^{31/32},
\end{align}
as well as that coming from the classical bound on bilinear sums,
\begin{align}
T(\cW, \cX,\cY,\cZ;\vartheta, \rho,\sigma,\tau) \ll p^{1/2}W^{1/2}X^{1/2}YZ.
\end{align}
For $W=X=Y=Z$ Theorem, \ref{thm:quad} is stronger than the classical bound and \eqref{eq:quadPetShp} for all $p^{1/2}<W<p^{13/24}$, in this range it is also stronger than the bound of Petridis and Shparlinski \cite{PetShp}. We give another example for when our bound is non-trivial. Setting $W=p^{2/3}, X=Y=Z=p^{3/8}$ we obtain from Theorem \ref{thm:quad}
$$
T(\cW, \cX,\cY,\cZ;\vartheta, \rho,\sigma,\tau) \ll p^{1355/768}=WXYZp^{-7/256}.
$$
We also mention that our bound is strongest for $W$ much larger than $X$.

\section{Incidence Results}
\subsection{Preliminaries}
In this section we use $\ell$ to indicate all possible lines.

We mention the following results.
\begin{lemma} \label{lem:iab}
Let $\cA,\cB \in \F_p$ with $|\cA|=A, |\cB|=B$ and $\lambda, \mu \in \F^*_p$. Then
\begin{align*}
\sum_\ell \iota_{\cA\times \cB}(\ell_{a,b})=\sum_\ell \iota_{\cA\times \cB}(\ell_{\lambda a,\mu b}) = pAB 
\end{align*}
and
\begin{align*}
\sum_\ell \iota_{\cA\times \cB}(\ell_{a, b}) \iota_{\cA\times \cB}(\ell_{\lambda a,\lambda b}) = A^2B^2-AB^2+pAB
\end{align*}
\end{lemma}
\begin{proof}
The first result is clear since for each choice of $(x,y,u) \in \cA\times \cB \times \F_p$ there is a unique choice of $v\in \F_p$.
The second result we have
\begin{align*}
&\sum_\ell \iota_{\cA\times \cB}(\ell_{a, b}) \iota_{\cA\times \cB}(\ell_{\lambda a,\lambda b})\\
&\qquad = \sum_{(a_1,a_2,b_1,b_2)\in \cA^2 \times \cB^2} |\{(c,d)\in \F^2_p : b_1=ca_1+d, b_2=\lambda ca_2+\mu d\}|.
\end{align*}
Now there are $AB$ quadruples $(a_1,a_2,b_1,b_2)\in \cA^2\times\cB^2$ with $(a_1,b_1)=(\lambda\mu^{-1} a_2,\mu^{-1} b_2)$ which define $p$ pairs $(c,d)=(c,b_1-ca_1)$. There are $AB(B-1)$ quadruples with $b_1\neq \mu^{-1}b_2$ and $a_1=\lambda \mu^{-1}a_2$ which do not define any pairs $(c,d)$, as they are parallel. The remaining 
\begin{align*}
A^2B^2-AB(B-1)-AB=A^2B^2-AB^2
\end{align*}
quadruples define one pair $(c,d)$ each, as they are the non-parallel lines.
\end{proof}
We immediately have the following corollary.
\begin{cor}\label{cor:iab}
Let $\cA,\cB \in \F_p$ with $|\cA|=A, |\cB|=B$ and $\lambda, \mu \in \F^*_p$. Then
\begin{align*}
\sum_\ell \left(\iota_{\cA\times \cB}(\ell_{\lambda a,\mu b})-\frac{AB}{p}\right)^2\le  pAB.
\end{align*}
\end{cor}
We need an analogue of \cite[Lemma 9]{MPR-NRS}. First we recall \cite[Theorem 7]{MPR-NRS}.
\begin{lemma} \label{lem:lines}
Let $\cA,\cB \subset \F_p$ with $|\cA|=A\le |\cB|=B$ and let $L$ be a collection of lines in $\F_p^2$. Assume that $A|L|\le p^2$. Then the number of incidences $I(P,L)$ between the point set $P = \cA \times \cB$ and $L$ is bounded by 
\begin{align*}
I(P,L) \ll A^{3/4}B^{1/2}|L|^{3/4} +|P|+|L|. 
\end{align*}
\end{lemma}
 We define $L_{N_{\lambda, \mu}}$ to be the collection of lines that are incident to between $N$ and $2N$ points, that is
 \begin{align*}
 L_{N_{\lambda, \mu}} = \{\ell_{\lambda a,\mu b} \in L : N < \iota_{A\times B}(\ell_{\lambda a, \mu b}) \le 2N\}
 \end{align*}
for $\lambda, \mu \in \F^*_p$. We then have the following lemma.
\begin{lemma} \label{lem:LN}
Let $\cA, \cB \subset \F_p$ with $|\cA|=A\le |\cB|=B$, $\lambda, \mu \in \F^*_p$ and let $2AB/p \le N \le A$ be an integer greater than 1. Then
\begin{align*}
|L_{N_{\lambda, \mu}}| \ll \min\(\frac{ pAB}{N^2}, \frac{A^3B^2}{N^4}\).
\end{align*} 
\end{lemma}
\begin{proof}
Since $\iota_{\cA\times \cB}(\ell_{\lambda a,\mu b}) \ge 2AB/p$, for $\iota_{\cA\times \cB}(\ell_{\lambda a,\mu b}) \in L_{N_{\lambda, \mu}}$, we have
\begin{align*}
\iota_{\cA\times \cB}(\ell_{\lambda a,\mu b}) -AB/p \ge AB/p\ge N/2.
\end{align*}
Therefore, using Lemma \ref{cor:iab},
\begin{equation} \label{eq:pAB}
\begin{split}
\frac{N^2}{4}|L_{N_{\lambda, \mu}}| &\le \sum_{\iota_{\cA\times \cB}(\ell_{\lambda a,\mu b})\in L_{N_{\lambda, \mu}}} (\iota_{\cA\times \cB}(\ell_{\lambda a,\mu b})-AB/p)^2 \\
&\le \sum_l (\iota_{\cA\times \cB}(\ell_{\lambda a,\mu b})-AB/p)^2 \\
& \le pAB.
\end{split}
\end{equation}
Now suppose $2AB/p \le N <2AB^{1/2}/p^{1/2}$. From \eqref{eq:pAB}
\begin{align*}
|L_{N_{\lambda, \mu}}| \ll \frac{pAB}{N^2} < \frac{pAB}{N^2} \times \frac{4A^2B}{N^2p} = \frac{4A^3B^2}{N^4}.
\end{align*}
We now suppose $N \ge 2AB^{1/2}/p^{1/2}$. Now $N \ge 2AB/p$ hence by \eqref{eq:pAB} $L_{N_{\lambda, \mu}} \le 4pAB/N^2 \le p^2/A$. We can now apply Lemma \ref{lem:lines} to obtain
\begin{align*}
N|L_{N_{\lambda, \mu}}| \ll A^{3/4}B^{1/2}|L_{N_{\lambda, \mu}}|^{3/4}+AB+|L_{N_{\lambda, \mu}}|.
\end{align*}
We now observe when each term dominates, omitting the last term as it gives $N\ll 1$, to get
\begin{align*}
|L_{N_{\lambda, \mu}}| \ll \frac{A^3B^2}{N^4}+\frac{AB}N.
\end{align*}
We now recall $N\le A$, hence
\begin{align*}
|L_{N_{\lambda, \mu}}| \ll \frac{A^3B^2}{N^4}.
\end{align*}
This completes the proof.
\end{proof}
We now need the following lemma.
\begin{lemma} \label{lem:iab2}
For $\cA, \cB \subset \F_p$ with $|\cA|=A<|\cB|=B$ and $\lambda, \mu \in \F^*_p$,
\begin{align*}
\sum_\ell \iota_{\cA\times \cB}(\ell_{a,b})\left(\iota_{\cA\times \cB}(\ell_{\lambda a,\mu b})-\frac{AB}{p}\right)^2 \ll p^{1/2}A^2B^{3/2}.
\end{align*}
\end{lemma}
\begin{proof}
We begin by splitting our sum over a parameter $\Delta$ which will be chosen later. We also observe that $\iota_{\cA\times \cB}(\ell_{\lambda a,\mu b}) \le A$. We then find a bound on
\begin{align*}
&\sum_{\iota_{\cA\times \cB}(\ell_{a,b})\le \Delta} \iota_{\cA\times \cB}(\ell_{a,b})\left(\iota_{\cA\times \cB}(\ell_{\lambda a,\mu b})-\frac{AB}{p}\right)^2 \\
&\qquad + \sum_{\substack{\iota_{\cA\times \cB}(\ell_{a,b})> \Delta \\ \iota_{\cA\times \cB}(\ell_{\lambda a,\mu b})\le \Delta }} \iota_{\cA\times \cB}(\ell_{a,b})\left(\iota_{\cA\times \cB}(\ell_{\lambda a,\mu b})-\frac{AB}{p}\right)^2\\ 
&\qquad + \sum_{\substack{\iota_{\cA\times \cB}(\ell_{a,b})> \Delta \\ \iota_{\cA\times \cB}(\ell_{\lambda a,\mu b})> \Delta }} \iota_{\cA\times \cB}(\ell_{a,b})\left(\iota_{\cA\times \cB}(\ell_{\lambda a,\mu b})-\frac{AB}{p}\right)^2 = I + II+III.
\end{align*}
By using Corollary \ref{cor:iab} it is clear that $I \le \Delta pAB$. We also have 
\begin{align*}
II \le  \sum_{\iota_{\cA\times \cB}(\ell_{a,b})> \Delta} \iota_{\cA\times \cB}(\ell_{a,b})\left(\Delta-\frac{AB}{p}\right)^2.
\end{align*}
Using dyadic decomposition and Lemma \ref{lem:LN} we obtain
\begin{align*}
II &\ll \left(\Delta-\frac{AB}{p}\right)^2 \sum_{k\ge 0} (2^k\Delta)|L_{2^k\Delta}| \\
&\ll \left(\Delta-\frac{AB}{p}\right)^2  \sum_{k\ge 0}(2^k\Delta) \frac{A^3B^2}{(\Delta 2^k)^4}\\
& \ll \left(\Delta-\frac{AB}{p}\right)^2  \frac{A^3B^2}{\Delta^3}.
\end{align*}

 From \eqref{i3}, for  $\Delta>2AB/p$, we have
\begin{align*}
& \sum_{\substack{\iota_{\cA\times \cB}(\ell_{a,b})> 2AB/p \\ \iota_{\cA\times \cB}(\ell_{\lambda a,\mu b})> 2AB/p}}\iota_{\cA\times \cB}(\ell_{a,b})\iota_{\cA\times \cB}(\ell_{\lambda a,\mu b})^2 \\
&\qquad \ge \sum_{\substack{\iota_{\cA\times \cB}(\ell_{a,b})> 2AB/p \\ \iota_{\cA\times \cB}(\ell_{\lambda a,\mu b})> 2AB/p}} \iota_{\cA\times \cB}(\ell_{a,b})\left(\iota_{\cA\times \cB}(\ell_{\lambda a,\mu b})-\frac{AB}{p}\right)^2\\
 & \qquad \qquad \qquad + \frac{3A^2B^2}{p^2}\sum_{\iota_{\cA\times \cB}(\ell_{a,b})> \frac{2AB}{p}} \iota_{\cA\times \cB}(\ell_{a,b}) \\
 & \qquad \ge \sum_{\substack{\iota_{\cA\times \cB}(\ell_{a,b})> 2AB/p \\ \iota_{\cA\times \cB}(\ell_{\lambda a,\mu b})> 2AB/p}} \iota_{\cA\times \cB}(\ell_{a,b})\left(\iota_{\cA\times \cB}(\ell_{\lambda a,\mu b})-\frac{AB}{p}\right)^2=III.
\end{align*}
We can now use dyadic decomposition and Lemma \ref{lem:LN} to obtain
\begin{align*}
\sum_{\substack{\iota_{\cA\times \cB}(\ell_{a,b})> \Delta \\ \iota_{\cA\times \cB}(\ell_{\lambda a,\mu b})> \Delta }} \iota_{\cA\times \cB}(\ell_{a,b})\iota_{\cA\times \cB}(\ell_{\lambda a,\mu b})^2  &\ll \sum_{k\ge 0} (2^k\Delta)^3|L_{2^k\Delta}| \\
&\ll \sum_{k\ge 0}(2^k\Delta)^3 \frac{A^3B^2}{(\Delta 2^k)^4}\\
& \ll \frac{A^3B^2}{\Delta}.
\end{align*}
Therefore,
\begin{align*}
\sum_\ell \iota_{\cA\times \cB}(\ell_{a,b})\left(\iota_{\cA\times \cB}(\ell_{\lambda a,\mu b})-\frac{AB}{p}\right)^2 \ll \Delta pAB +II+ \frac{A^3B^2}{\Delta}.
\end{align*}
We choose $\Delta =AB^{1/2}/p^{1/2}$ to get
\begin{align*}
II &\ll \frac{A^2B}{p}\left(1- \frac{B^{1/2}}{p^{1/2}} \right) p^{3/2}B^{1/2}\\
&\ll p^{1/2}A^2B^{3/2}
\end{align*}
and
\begin{align*}
\sum_\ell \iota_{\cA\times \cB}(\ell_{a,b})\left(\iota_{\cA\times \cB}(\ell_{\lambda a,\mu b})-\frac{AB}{p}\right)^2 \ll p^{1/2}A^2B^{3/2},
\end{align*}
assuming $AB^{1/2}/p^{1/2} \ge \frac{2AB}p$. Otherwise $p < 4B$, but then it is clear from Corollary \ref{cor:iab} that
\begin{align*}
\sum_{\ell} \iota_{\cA\times \cB}(\ell_{a,b})\left(\iota_{\cA\times \cB}(\ell_{\lambda a,\mu b})-\frac{AB}{p}\right)^2 \le pA^2B \ll p^{1/2}A^2B^{3/2}.
\end{align*}
This completes the proof.
\end{proof}
\subsection{Proof of Theorem \ref{thm:TAB}}
We can transform $T_{\lambda, \mu}(\cA,\cB)$ to be the number of solutions of
\begin{align*}
\frac{b_1-\mu b_2}{a_1-\lambda a_3}=\frac{b_1-\mu b_3}{a_1-\lambda a_2},
\end{align*}
by adding an error term of $O(AB^3+A^2B^2)$ coming from the cases where $a_1=\lambda a_2=\lambda a_3$, or $a_1=\lambda a_3$ and $b_1=\mu b_2$, or $a_1=\lambda a_2$ and $b_1=\mu b_3$. Then collecting our solutions for each $c \in \F_p$,
\begin{align*}
\frac{b_1-\mu b_2}{a_1-\lambda a_3}=\frac{b_1-\mu b_3}{a_1-\lambda a_2}=c
\end{align*}
and re-arranging and relabelling, we obtain
\begin{align*}
b_1-ca_1=\mu b_2-\lambda ca_2=\mu b_3-\lambda ca_3.
\end{align*}
Therefore,
\begin{align} \label{eq:TABi}
T_{\lambda, \mu}(\cA,\cB)=\sum_\ell \iota_{\cA\times \cB}(\ell_{a,b})\iota_{\cA\times \cB}(\ell_{\lambda a,\mu b})^2 +O(AB^3+A^2B^2).
\end{align}

As in \cite[p. 6]{MaShkShp}, we use the result $X^2=(X-Y)^2+2XY-Y^2$ with $X=\iota_{\cA\times \cB}(\ell_{\lambda a,\mu b})$ and $Y=AB/p$ and see
\begin{equation} \label{i3}
\begin{split}
&\sum_\ell \iota_{\cA\times \cB}(\ell_{a,b})\iota_{\cA\times \cB}(\ell_{\lambda a,\mu b})^2\\
& \qquad \qquad = \sum_\ell \iota_{\cA\times \cB}(\ell_{a,b})\left(\iota_{\cA\times \cB}(\ell_{\lambda a,\mu b})-\frac{AB}{p}\right)^2 \\
& \qquad \qquad \qquad + \frac{2AB}{p}\sum_\ell \iota_{\cA\times \cB}(\ell_{a,b})\iota_{\cA\times \cB}(\ell_{\lambda a,\mu b})-\frac{A^2B^2}{p^2}\sum_\ell \iota_{\cA\times \cB}(\ell_{a,b}).
\end{split}
\end{equation}
We now apply Lemma \ref{lem:iab} to obtain,
\begin{align*}
&\sum_\ell \iota_{\cA\times \cB}(\ell_{a,b})\iota_{\cA\times \cB}(\ell_{\lambda a,\mu b})^2  \\
&\quad = \sum_\ell \iota_{\cA\times \cB}(\ell_{a,b})\left(\iota_{\cA\times \cB}(\ell_{\lambda a,\mu b})-\frac{AB}{p}\right)^2 +\frac{A^3B^3-2A^3B^2}{p}+2A^2B^2.
\end{align*}
Combining \eqref{eq:TABi}, \eqref{i3} and Lemma \ref{lem:iab2} we complete the proof.
\subsection{Consequences}
We give some results that come as a consequence of Theorem \ref{thm:TAB}, these are necessary for our proofs of Theorem \ref{thm:tri} and Theorem \ref{thm:quad}.

We define $D_{\lambda, \mu}(\cA,\cB)$ to be the number of solutions to
\begin{align} \label{eq:Dtimes}
(a_1-\lambda a_2)(b_1-\mu b_2)=(a_3-\lambda a_4)(b_3-\mu b_4)
\end{align}
for $(a_i,b_i) \in \cA\times \cB, i=1,2,3,4$, and $\lambda, \mu \in \F^*_p$. We also define $T_{\lambda,\mu}^*(\cA,\cB)$ to be the number of solutions of
\begin{align*}
(a_1-\lambda a_2)(b_1-\mu b_2)=(a_1-\lambda a_3)(b_1-\mu b_3)\neq 0
\end{align*}
and, similarly, $D^*_{\lambda, \mu}(\cA,\cB)$ to be the number of solutions of
\begin{align*}
(a_1-\lambda a_2)(b_1-\mu b_2)=(a_3-\lambda a_4)(b_3-\mu b_4)\neq 0.
\end{align*}
We also define $D^*_{1,1}(\cA,\cB)=D^*(\cA,\cB)$, $D_{1,1}(\cA,\cB)=D(\cA,\cB)$ and $T_{1,1}^*(\cA,\cB)=T^*(\cA,\cB)$.
\begin{lemma} \label{lem:D*}
Let $\cA, \cB \subset \F_p$ with $|\cA|=A\le |\cB|=B$ and $\lambda, \mu \in \F^*_p$. Then
\begin{align*}
D^*_{\lambda, \mu}(\cA,\cB)\ll p^{1/2}A^3B^{5/2}+\frac{A^4B^4}p.
\end{align*}
\end{lemma}
\begin{proof}
We rearrange $D^*_{\lambda,\mu}(\cA,\cB)$ so it is the number of solutions of 
\begin{align*}
\frac{b_1-\mu b_2}{a_3-\lambda a_4}=\frac{b_3-\mu b_4}{a_1-\lambda a_2}\neq 0.
\end{align*}
We define $J(\xi)$ to be the number of quadruples $(a_1,a,b_1,b)\in \cA^2\times \cB^2$ with
\begin{align} \label{eq:Jlambda}
\frac{b-\mu b_1}{a-\lambda a_1}=\xi.
\end{align}
We also let $J_{a,b}(\xi)$ be the number of pairs $(a,b) \in \cA \times \cB$ for which \eqref{eq:Jlambda} holds. Then by the Cauchy inequality, we have
\begin{align*}
D^*_{\lambda,\mu}(\cA,\cB)&= \sum_{\xi \in \F^*_p}J(\xi)^2= \sum_{\xi \in \F^*_p}\(\sum_{(a,b)\in \cA\times \cB}J_{a,b}(\xi)\)^2 \\
&\le AB \sum_{\xi \in \F^*_p}\sum_{(a,b)\in \cA\times \cB}J_{a,b}(\xi)^2= AB \sum_{(a,b)\in \cA\times \cB}\sum_{\xi \in \F^*_p}J_{a,b}(\xi)^2.
\end{align*}
Now
\begin{align*}
\sum_{\xi \in \F^*_p}J_{a,b}(\xi)^2 = \|\{(a_1,a_2,b_1,b_2)\in \cA^2\times \cB^2 : \frac{b-\mu b_1}{a-\lambda a_1}=\frac{b-\mu b_2}{a-\lambda a_2}\neq 0\}\|,
\end{align*}
hence
\begin{align*}
D^*_{\lambda,\mu}(\cA,\cB) &\le AB T^*_{\lambda,\mu}(\cA,\cB)\le AB \sum_\ell \iota_{\cA\times \cB}(\ell_{a,b})\iota_{\cA\times \cB}(\ell_{\lambda a,\mu b})^2 \\
& \ll p^{1/2}A^3B^{5/2}+\frac{A^4B^4}p.
\end{align*}
This concludes the proof.
\end{proof}
Since the number of solutions for when \eqref{eq:Dtimes} is equal to $0$ is $O(A^2B^4+A^3B^3+A^4B^2)$ we get the following simple corollary.
\begin{cor}
Let $\cA, \cB \subset \F_p$ with $|\cA|=A\le |\cB|=B$ and $\lambda,\mu \in \F^*_p$. Then
\begin{align*}
D_{\lambda,\mu}(\cA,\cB)\ll p^{1/2}A^3B^{5/2}+\frac{A^4B^4}p+A^2B^4.
\end{align*}
\end{cor}

\section{Exponential Sums}
\subsection{Preliminaries}
We recall the classical bound for bilinear exponential sums, see \cite[Equation 1.4]{BouGar} or \cite[Lemma 4.1]{Gar}.
\begin{lemma}
\label{lem:bilin} 
For any sets $\cX, \cY \subseteq \F_p$ and any  $\alpha= (\alpha_{x})_{x\in \cX}$, $\beta = \( \beta_{y}\)_{y \in \cY}$ with 
\begin{align*}
\sum_{x\in \cX}|\alpha_{x}|^2 = A \mand  \sum_{y \in \cY}|\beta_{y}|^2 = B, 
\end{align*}
we have 
\begin{align*}
\left |\sum_{x \in \cX}\sum_{y \in \cY} \alpha_{x} \beta_{y}  \ep(xy) \right| \le \sqrt{pAB}.
\end{align*}
\end{lemma}
We define $N(\cX,\cY,\cZ)$ to be the number of solutions to
\begin{align*}
x_1(y_1-z_1)=x_2(y_2-z_2)
\end{align*}
with $x_1,x_2 \in \cX, y_1, y_2 \in \cY$ and $z_1,z_2 \in \cZ$. We now recall \cite[Corollary 2.4]{PetShp}.
\begin{lemma} \label{lem:NXYZ}
Let $\cX, \cY, \cZ \subset \F^*_p$ with $|\cX|=X, |\cY|=Y, |\cZ|=Z$ and $M=\max(X,Y,Z)$. Then
\begin{align*}
N(\cX,\cY,\cZ) \ll \frac{X^2Y^2Z^2}p + X^{3/2}Y^{3/2}Z^{3/2}+MXYZ.
\end{align*}
\end{lemma}

\subsection{Proof of Theorem \ref{thm:tri}}
 We use Lemma \ref{lem:D*} in the proof of \linebreak \cite[Theorem 1.3]{PetShp} to give a new bound on trilinear exponential sums. We pick up the proof of \cite[Theorem 1.3]{PetShp} at equation (3.8), permuting the variables we obtain
\begin{align*}
T(\cX,\cY,\cZ;\rho,\sigma,\tau)^8 \ll pX^4Y^7Z^4K+X^8Y^8Z^6.
\end{align*}
Now $K$ is simply $D^*(\cX,\cZ)$ hence by Lemma \ref{lem:D*}
\begin{align} \label{eq:T38}
T(\cX,\cY,\cZ;\rho,\sigma,\tau)^8 \ll p^{3/2}X^{13/2}Y^{7}Z^7+X^8Y^7Z^8+X^8Y^8Z^6.
\end{align}
We then take $8$th roots and compare with the classical bound on bilinear sums, Lemma \ref{lem:bilin}, combined with the triangle inequality
$$
T(\cX,\cY,\cZ;\rho,\sigma,\tau) \ll p^{1/2}X^{1/2}Y^{1/2}Z.
$$
For our bound to be non-trivial
$$
p^{3/16}X^{13/16}Y^{7/8}Z^{7/8} \le p^{1/2}X^{1/2}Y^{1/2}Z,
$$
or equivalently
$$
X^{5/16}Y^{3/8}Z^{-1/8} \le p^{5/16},
$$
therefore,
$$
XY^{4/5} \le p.
$$
Now for $XY^{4/5} \le p$ we have
$$
XY^{7/8}Z \le p^{3/16}X^{13/16}Y^{29/40}Z  \le  p^{3/16}X^{13/16}Y^{7/8}Z^{7/8}.
$$
Similarly, 
$$
XYZ^{3/4} \le p^{3/16}X^{13/16}Y^{34/40}Z^{3/4} \le  p^{3/16}X^{13/16}Y^{7/8}Z^{7/8}.
$$
Hence our first term dominates over the non-trivial region. Furthermore, when our bound is trivial, i.e. for $X^{5/16}Y^{3/8}Z^{-1/8} \le p^{5/16}$,
$$
T(\cX,\cY,\cZ;\rho,\sigma,\tau) \ll p^{1/2}X^{1/2}Y^{1/2}Z \ll p^{3/16}X^{13/16}Y^{7/8}Z^{7/8}.
$$
This concludes the proof.
\subsection{Proof of Theorem \ref{thm:quad}}
 
We use Lemma \ref{lem:D*} in the proof of \linebreak \cite[Theorem 1.4]{PetShp} to give a new bound on weighted quadrilinear exponential sums. As in the proof of \cite[Theorem 1.4]{PetShp}, after permuting the variables, we have
\begin{equation} \label{eq:T4}
\begin{split}
&T(\cW,\cX,\cY,\cZ;\vartheta,\rho,\sigma,\tau)^8 \\
&\qquad \ll (WXY)^6Z^7\sum_{\mu \in \F^*_p}\sum_{\lambda \in \F_p}J(\mu)\eta_\mu I(\lambda)\ep(\lambda\mu) + (WXZ)^8Y^7,
\end{split}
\end{equation}
where $I(\lambda)$ is the number of triples $(x_1,x_2,z) \in \cX^2\times \cZ$ with $z(w_1-w_2)=\lambda$, $J(\mu)$ is the number of quadruples $(w_1,w_2,y_1,y_2)\in \cW^2 \times \cY^2$ with $(w_1-w_2)(y_1-y_2)=\mu$ and $\eta_\mu$ is a complex number with $|\eta_\mu|=1$. It is clear that
\begin{align*}
\sum_{\mu \in \F^*_p} J(\mu)^2 = D^*(\cW,\cY) \ll p^{1/2}W^{5/2}Y^3+\frac{W^4Y^4}p.
\end{align*}
We now use Lemma \ref{lem:NXYZ} to obtain
\begin{align*}
\sum_{\lambda \in \F_p} I(\lambda)^2 \ll \frac{Z^2X^4}p + Z^{3/2}X^3+ZX^3\ll\frac{X^4Z^2}p + X^3Z^{3/2} .
\end{align*}
We now apply the classical bound for bilinear sums, Lemma \ref{lem:bilin}, to \eqref{eq:T4} to obtain
\begin{align*}
&T(\cW,\cX,\cY,\cZ;\vartheta,\rho,\sigma,\tau)^8 \\
&\quad \ll (WXY)^6Z^7\bigg(p^{1/4}W^{5/4}Y^{3/2}+ \frac{W^2Y^2}{p^{1/2}}\bigg)\bigg(p^{1/2}X^{3/2}Z^{3/4}+X^2Z\bigg) \\
&\qquad \qquad \qquad \qquad \qquad \qquad \qquad \qquad \qquad \qquad \qquad + (WXZ)^8Y^7.
\end{align*}
We compare the above bound with the classical bound on bilinear sums combined with the triangle inequality
 $$T(\cW,\cX,\cY,\cZ;\vartheta,\rho,\sigma,\tau)^8\ll p^{4}W^{4}X^{4}Y^8Z^8$$
coming from Lemma \ref{lem:bilin}. For our bound to be non-trivial we need 
$$p^{3/4}W^{29/4}X^{15/2}Y^{15/2}Z^{31/4} \le p^{4}W^{4}X^{4}Y^8Z^8.$$
That is,
$$
W^{13/4}X^{7/2}Y^{-1/2}Z^{-1/4} \le p^{13/4},
$$
therefore, since $Z\le Y\le X$
$$
WX^{11/13}\le p.
$$
Now for $WX^{11/13}\le p$,
$$
X^2Z \le p^{13/48}X^{3/2}Z\le p^{13/32}X^{3/2}Z^{3/4}<p^{1/2}X^{3/2}Z^{3/4}.
$$
Similarly,
$$
\frac{W^2Y^2}{p^{1/2}} \le \frac{p^{3/4}W^{5/4}Y^2}{X^{33/52}p^{1/2}}\le p^{1/4}W^{5/4}Y^{71/52} \le p^{1/4}W^{5/4}Y^{3/2}.
$$
Finally,
\begin{align*}
(WXZ)^8Y^7\le p^{3/4}W^{29/4}X^{383/52}Y^7Z^8&\le p^{3/4}W^{29/4}X^{15/2}Y^7Z^8 \\
&\le p^{3/4}W^{29/4}X^{15/2}Y^{15/2}Z^{31/4}.
\end{align*}
Hence, for  $WX^{11/13}\le p$, after taking $8$th roots
\begin{align*}
T(\cW,\cX,\cY,\cZ;\vartheta,\rho,\sigma,\tau) \ll  p^{3/32}W^{29/32}X^{15/16}Y^{15/16}Z^{31/32}.
\end{align*}
However, for $WX^{11/13} > p$, then our bound is trivial and 
\begin{align*}
T(\cW,\cX,\cY,\cZ;\vartheta,\rho,\sigma,\tau) &\ll p^{1/2}W^{1/2}X^{1/2}YZ\\
&\ll p^{3/32}W^{29/32}X^{15/16}Y^{15/16}Z^{31/32}.
\end{align*}
This completes the proof.


\section{Acknowledgements}
The author is very thankful to Giorgis Petridis for his many suggestions. The author is also thankful to Igor Shparlinski for his suggestions and proof-reading of the paper.

\end{document}